\documentclass{article}
\usepackage[utf8]{inputenc}
\usepackage[height=8.2in,width=5.8in]{geometry}
\usepackage{amsmath,amsthm,amssymb,bm}
\usepackage{graphicx,mathrsfs}
\usepackage{natbib}
\usepackage{enumerate}
\usepackage[colorlinks=true,citecolor=blue]{hyperref}
\usepackage{microtype}

\newtheorem{theorem}{Theorem}[section]

\newtheorem{prop}[theorem]{Proposition}
\newtheorem{corr}[theorem]{Corollary}
\newtheorem{conj}[theorem]{Conjecture}

\newtheorem{rmk}[theorem]{Remark}

\newcommand{\nbk}[2]{[#1\backslash #2]}
\newcommand{\dxg}{d_\chi^G}

\newcommand{\kron}{\otimes}
\newcommand{\nlsum}{\sum\nolimits}

\newcommand{\set}[1]{\{ #1\}}

\newcommand{\norm}[1]{\Vert#1\Vert}

\newcommand{\Tc}{\mathcal{T}}
\newcommand{\Rc}{\mathcal{R}}
\newcommand{\Lc}{\mathcal{L}}
\newcommand{\sep}{;}
\date{Last edited: 12 Nov, 2014}

\numberwithin{equation}{section}

\begin{document} 
\title{Hlawka-Popoviciu inequalities on positive definite tensors}
\author{Wolfgang Berndt\and Suvrit Sra}
\maketitle

\begin{abstract}
  We prove inequalities on symmetric tensor sums of positive definite operators. In particular, we prove multivariable operator inequalities inspired by generalizations to the well-known Hlawka and Popoviciu inequalities. As corollaries, we obtain generalized Hlawka and Popoviciu inequalities for determinants, permanents, and generalized matrix functions. The new operator inequalities and their corollaries contain a few recently published inequalities on positive definite matrices as special cases.
\end{abstract}

 
%\MSC 15A15 \sep 15A39 \sep 15A69 \sep 46M05 \sep 47A63

\section{Introduction}
Let $X$ be a complex inner product space with norm $\norm{\cdot}$, and let $a, b, c \in X$ be arbitrary vectors. The inequality
\begin{equation}
  \label{eq:10}
  \norm{a+b+c} + \norm{a}+\norm{b}+\norm{c} \geqslant \norm{a+b} + \norm{a+c} + \norm{b+c},
\end{equation}
is known as \emph{Hlawka's inequality}. It seems to have appeared first in a paper of Hornich~\citep{hornich} (who credits the proof to Hlawka). Several proofs are known, see e.g.~\citep[p.~100]{niculesu06} or \citep[pp.~171-72]{mitrinovic}. This inequality has witnessed a long series of investigations and generalizations---we refer the reader to the recent work of Fechner~\citep{fechner} for an excellent summary of related work as well as a substantial list of references. Fechner himself considers the \emph{functional Hlawka inequality}
\begin{equation}
 \label{eq:11}
  f(a+b+c) + f(a) + f(b) + f(c) \geqslant f(a+b) + f(a+c) + f(b+c),
\end{equation}
 and studies real valued functions $f$ on an abelian group $(A,+)$ that satisfy \eqref{eq:11}.

To our knowledge, all authors who previously published Hlawka type inequalities limited their attention to inequalities over the reals. In contrast, we study ``operator Hlawka inequalities,'' so that instead of the total order on the reals, we consider the L\"owner partial order `$\geqslant$' on Hermitian positive definite matrices or operators. As a consequence, we are able to recover as corollaries several Hlawka type inequalities for scalar valued matrix functions known as ``generalized matrix functions,'' which include the determinant and permanent as special cases.  

Observe also the resemblance between \eqref{eq:11} and \emph{Popoviciu's inequality}, which states for a convex function $f$ on a real interval $I$ and $a,b,c \in I$ that
\begin{equation}
  \label{eq:pop01}
  3f\bigl(\tfrac{a+b+c}{3}\bigr)+f(a)+f(b)+f(c)\geqslant
  2\bigl( f\bigl(\tfrac{a+b}{2}\bigr) +f\bigl(\tfrac{a+c}{2}\bigr) +f\bigl(\tfrac{b+c}{2}\bigr)\bigr). 
\end{equation}  
In fact, this resemblance will allow us to obtain some operator Popoviciu inequalities.

\subsection*{Notation and Background} Throughout this paper, matrices and tensors are denoted by upper case letters. Unless otherwise specified, all matrices are assumed to be of same size (say $m\times m$), self-adjoint and positive (semi) definite. The operator inequality $A \geqslant B$ denotes the L\"owner partial order, meaning that $A-B \geqslant 0$ is positive definite. Wherever multiplication is used, we mean tensor products (though unusual, we use this notation for aesthetic reasons to keep the ``visual burden'' of our proofs low); thus for arbitrary matrices $A$, $B$:
\begin{align*}
  A^p    &\equiv A^{\kron p} = A\kron A \kron \cdots \kron A\quad (p\ \ \text{times})\\
  A^pB^q &\equiv (A^{\kron p}) \kron (B^{\kron q})\qquad\quad\qquad(\text{integers\ } p, q).
\end{align*}
Note that this multiplication is noncommutative, so $AB \neq BA$.   

\noindent We write $[n]$ to denote the set $\set{1,2,\ldots,n}$ and $\nbk{n}{k}$ to denote the set $[n] \setminus \set{k}$.     

\noindent For some indexes, we use \textsc{Matlab} notation, e.g. the form $i=1\!:\!2\!:\!2k-1$ meaning that $i$ ``steps by 2,'' taking on only the values $1,3,5,\ldots,2k-1$.

Since the entire paper relies extensively on elementary properties of Kronecker (tensor) products, let us briefly recall these below.
\begin{prop}
  \label{prop.basic}
  Let $A, B, C, D$ be positive definite operators. Then,\vspace*{-5pt}
  \begin{enumerate}[(i)]
    \setlength{\itemsep}{-1pt}
  \item $AB \equiv A\kron B$ is also positive definite
  \item If $A \geqslant B$ and $C \geqslant D$ then $AC \geqslant BD$
  \item $A(B+C) = AB + AC$, $(A+B)C = AC + BC$
  \item $(A+B)^p \geqslant A^p + B^p$ for all $p \in \mathbb{N}$.
  \end{enumerate}
\end{prop}

\section{ Hlawka type inequalities for three operators}
With this background we are ready to prove our first operator Hlawka inequality.
\begin{theorem}
  \label{thm:one}
  Let $A, B, C$ be positive definite operators. Then for each integer $p \geqslant 1$,
  \begin{equation}
    \label{eq:1}
    (A+B+C)^p + A^p + B^p + C^p \geqslant (A+B)^p + (A+C)^p + (B+C)^p.
  \end{equation}
\end{theorem}
\begin{proof}
  The case $p=1$ is trivial and holds with equality. Unsurprisingly, for $p=2$ we again have equality, since both sides expand to
  \begin{align*}
   2( A^{2} + B^{2} + C^{2} ) + AB + BA + AC + CA + BC + CB.
  \end{align*}
We prove the general claim by induction. Assume therefore that~\eqref{eq:1} holds for some integer $p \geqslant 2$. Then, 
  \begin{align*}
    (A+B+C)^{p+1} &= (A+B+C)^p(A+B+C)\\
    &\geqslant \left((A+B)^p + (A+C)^p + (B+C)^p - A^p - B^p - C^p \right)(A+B+C)\\
    &= (A+B)^{p+1} + (A+C)^{p+1} + (B+C)^{p+1} - A^{p+1} - B^{p+1} - C^{p+1} + \Tc,
  \end{align*}
  where the inequality follows from the induction hypothesis. The term $\Tc$ is defined as
  \begin{equation*}
    \Tc = (A+B)^pC + (A+C)^pB + (B+C)^pA - A^p(B+C) - B^p(A+C) - C^p(A+B).
  \end{equation*}
  It remains to show that $\Tc \geqslant 0$. But this follows immediately upon applying the superadditivity inequality Prop.~\ref{prop.basic}(iv) 
    to the first three terms of $\Tc$ and canceling. Thus, inequality~\eqref{eq:1} is proved.
\end{proof}

Theorem~\ref{thm:one} yields the following result of~\citet[Lemma~2.2]{tie11} as a corollary. (Note that the inequality~\eqref{eq:super} is called \emph{strong superadditivity} of tensor products; readers familiar with combinatorics may recognize it as  \emph{supermodularity}).
\begin{corr}
  \label{cor.supermod}
  Let $A,B,C$ be positive definite operators. Then for each integer $p \geqslant 1$;
  \begin{equation}
    \label{eq:super}
    (A+B+C)^p +A^p \geqslant (A+B)^p + (A+C)^p.
  \end{equation}
\end{corr}
\begin{proof}
  Immediate upon combining Prop.~\ref{prop.basic}(iv) with inequality~\eqref{eq:1}.
\end{proof}

Using the operator inequality~\eqref{eq:1} and restricting to suitable symmetry classes we can obtain Hlawka inequalities for determinants, permanents, and immanants. This line of thought is well-known in matrix analysis, see e.g.~\citep[p.~114]{bhatia07} and also~\citep{fzhang14}. 

Specifically, let $G$ be a subgroup of the symmetric group $\mathfrak{S}_m$ on $m$ letters, and let $\chi$ be an irreducible character of $G$. The $G$-\emph{immanant} (also known as \emph{generalized matrix function}~\citep{marcusv1,merris}) of an arbitrary $m \times m$ complex matrix $X$ is defined as
\begin{equation}
  \label{eq:2}
  d_{\chi}^G(X) := \sum_{\sigma \in G}\chi(\sigma)\prod_{i=1}^m a_{i,\sigma(i)}.
\end{equation}

\noindent When $G=\mathfrak{S}_m$ and $\chi(\sigma)=\text{sgn}(\sigma)$ we have $\dxg(X)=\det(X)$; $\chi(\sigma)\equiv 1$ yields the permanent, while other choices yield immanants~\citep{marcusv1}. Using arguments from multilinear algebra (e.g., \citep{marcusv1,ckli}), it can be shown \citep[p.~126]{marcusv1} that there exists a matrix $Z_{G,\chi}$ such that
\begin{equation}
  \label{eq:8}
  \dxg(X) = Z_{G,\chi}^*(\otimes^mX)Z_{G,\chi}.
\end{equation}

\noindent Using representation~\eqref{eq:8} and Theorem~\ref{thm:one} we then obtain the following corollary.
\begin{corr}
  \label{corr:one}
  Let $A, B, C$ be positive definite, and let $\dxg$ be as defined by~\eqref{eq:2}. Then,
  \begin{equation}
    \label{eq:3}
    \dxg(A+B+C) + \dxg(A)+\dxg(B) + \dxg(C) \geqslant \dxg(A+B) + \dxg(A+C) + \dxg(B+C).
  \end{equation}
\end{corr}
\begin{proof}
  Congruence preserves L\"owner order, so we use~\eqref{eq:8} and~\eqref{eq:1} and conclude.
\end{proof}

\begin{rmk}
  The recent strong superadditivity result of~\citet[Theorem~3.2]{fzhang14} for three matrices follows by combining \eqref{eq:8}  with Corollary~\ref{corr:one} and Prop.~\ref{prop.basic}(iv).
\end{rmk}
\begin{rmk}
  M.~Lin brought to our notice his very recent result that establishes inequality~\eqref{eq:3} for the special case of determinants~\citep{mlin}. His proof uses only elementary methods, is entirely different from our approach, and is of instructive value.
\end{rmk}

\section{A multivariable tensor Hlawka inequality}
\label{sec:tensor}
It turns out that the above results can be obtained as corollaries of a more general operator inequality involving $n$ positive definite matrices. Before considering this more general inequality, let us mention a Hlawka type inequality that was conjectured by the first named author, which originally inspired this paper.

\begin{conj}[Berndt]
  \label{conj:one}
  For $n\geqslant 3$, let $A_1,\ldots,A_n$ be positive definite; for each $k=1,\dots,n$, let $s_k$ be the \emph{elementary symmetric determinantal polynomial}
  \begin{equation}
    \label{eq:4}
    s_k := \sum_{1 \leqslant i_1 < i_2 < \cdots < i_k \leqslant n} \det(A_{i_1} + \cdots + A_{i_k}).
  \end{equation}
  Then, the following generalization of the Hlawka inequality holds:
  \begin{equation}
    \label{eq:9}
    s_n + s_{n-2} + \cdots \geqslant s_{n-1} + s_{n-3} + \cdots.
  \end{equation}
\end{conj}

Inequality~\eqref{eq:9} may come as a surprise to those who study Hlawka type inequalities. Indeed, \citet{freudenthal} considered generalizing the basic norm inequality~(\ref{eq:10}) to a form similar to~\eqref{eq:9}. Specifically, he asked whether for $n$ vectors $a_1,\ldots,a_n$ the inequality
\begin{equation*}
  \sum_{i=1}^n\norm{a_i} - \sum_{i<j}\norm{a_i+a_j} \pm \cdots + (-1)^{n-1}\norm{a_1+\cdots+a_n} \geqslant 0 
\end{equation*}
holds. According to \citet[p. 174]{mitrinovic}, this inequality was shown to be false for $n\ge4$ by W.~A.~J.~Luxemburg. Nevertheless, other multivariable generalizations do hold, among which the following seems to be of the most general kind: 

\begin{prop}[{\citep[Corollary~3.5]{radu}}]
  \label{prop:radu}
Let $H$ be a metric space, $n \geqslant 3$ and $k\in\{2,  \dots , n \}$. Then for all $a_1,\dots ,a_n\in H$, 
\begin{equation}
\sum_{1 \leqslant i_1 < i_2 < \cdots < i_k \leqslant n} \norm{a_{i_1} + \cdots +a_{i_k}} \leqslant \binom{n-2}{k-1} \sum_{i=1}^n\norm{a_i}+ \binom{n-2}{k-2}\biggl\Vert{\sum_{i=1}^n a_i} \biggr\Vert.
\end{equation}
\end{prop}

We now proceed to show that for positive operators a multivariable Hlawka type inequality does hold. Combined with representation~\eqref{eq:8}, it then implies not only the determinantal inequality~\eqref{eq:9} but also its $G$-immanant version.

For positive integers $k,n,p$ with $k \le n$ define the following symmetric tensor sums:
\begin{equation}
  \label{eq:7}
  S_{k,[n]}^p := \sum_{\substack{I \subseteq [n], |I|=k}} \bigl(\nlsum_{i\in I}A_i\bigr)^p.
\end{equation}
The main result of this paper is the following theorem.

\begin{theorem}
  \label{thm:main}
  Let $n\geqslant 3$ and $A_1,\ldots,A_n \geqslant 0$. Then, for $p \in\mathbb N$ the operator inequality 
  \begin{equation}
  \label{eq:6}
  S_{n,[n]}^p + S_{n-2,[n]}^p + \cdots \geqslant S_{n-1,[n]}^p + S_{n-3,[n]}^p + \cdots
\end{equation} 
holds.
\end{theorem}
\begin{proof}
We prove the claim (call it $C_{n,p}$) by double induction. For $n=3$, $C_{3,p}$ is the Hlawka inequality established by Theorem~\ref{thm:one}. Fix $n\geq4$ and suppose we have proved $C_{n-1,p}$ for all $p$. We first assume that $n$ is even (the argument for odd $n$ will be similar).  

We now perform an induction on $p$. For $p=1$, the claim clearly holds as both sides of \eqref{eq:6} are equal. Assume therefore that the claim holds up to some integer $p-1$. Thus,
\begin{equation*}
  S_{n,[n]}^{p-1} + S_{n-2,[n]}^{p-1} + \cdots + S_{2,[n]}^{p-1}
  \geqslant S_{n-1,[n]}^{p-1} + S_{n-3,[n]}^{p-1} + \cdots + S_{1,[n]}^{p-1}.
\end{equation*}
Multiplying (i.e., taking tensor products) both sides by $(A_1+\cdots+A_n)$ on the right and using Prop.~\ref{prop.basic}(ii), we obtain
\begin{equation*}
  \sum_{j=2:2:n}S_{j,[n]}^p + \Lc
  \geqslant \sum_{j=1:2:n-1}S_{j,[n]}^p + \Rc,
\end{equation*}
where $\Lc$ and $\Rc$ denote the respective mixed terms. The claim $C_{n,p}$ will be proved if we show that $\Rc \geqslant \Lc$. Details follow below.

An easy rearrangement of the respective terms shows that
\begin{equation}
  \label{eq:5}
  \begin{split}
    \Lc &= \sum_{\substack{I\subset [n]\\ |I|=n-2}}\left(\nlsum_{i \in I}A_i\right)^{p-1}\left(\nlsum_{i\not\in I}A_i\right) + \cdots + \sum_{\substack{I\subset [n]\\|I|=2}}\left(\nlsum_{i\in I}A_i\right)^{p-1}\left(\nlsum_{i\not\in I}A_i\right)\\
    \Rc &= \sum_{\substack{I\subset [n]\\ |I|=n-1}}\left(\nlsum_{i \in I}A_i\right)^{p-1}\left(\nlsum_{i\not\in I}A_i\right) + \cdots + \sum_{k=1}^n A_k^{p-1}\left(\nlsum_{i\not=k}A_i\right)\\
  \end{split}
\end{equation}
Note that the main sums in $\Lc$ and $\Rc$ are only over even and odd sized subsets, respectively.  

The key to the proof is the following regrouping of~\eqref{eq:5}, which reveals the underlying inductive structure:
\begin{align*}
  &\Rc = \biggl(\sum_{i \in \nbk{n}{n}}A_i\biggr)^{p-1}A_n + 
      \biggl(\sum_{\substack{I\subset \nbk{n}{n}\\|I|=n-3 }}
        (\nlsum_{i\in I}A_i)^{p-1}\biggr)A_n+ \cdots
      + \biggl(\sum_{\substack{I\subset \nbk{n}{n}\\I=\{i\} }}A_i^{p-1}\biggr)A_n\; +\\
      &\biggl(\sum_{i \in \nbk{n}{n-1}}A_i\biggr)^{p-1}A_{n-1} + 
      \biggl(\sum_{\substack{I\subset \nbk{n}{n-1} \\ |I|=n-3\\ }}
      (\nlsum_{i\in I}A_i)^{p-1}\biggr)A_{n-1} + \cdots
      + \biggl(\sum_{\substack{ I\subset \nbk{n}{n-1} \\ I=\{i\} }}A_i^{p-1}\biggr)A_{n-1}\\
  &+\qquad\qquad\cdots\qquad\qquad +\qquad\qquad \cdots\qquad\qquad \qquad\cdots\qquad +\qquad\qquad \cdots\\
  &+\biggl(\sum_{i \in \nbk{n}{1}}A_i\biggr)^{p-1}A_{1} + 
      \biggl(\sum_{\substack{I\subset \nbk{n}{1}\\ |I|=n-3 }}
        (\nlsum_{i\in I}A_i)^{p-1}\biggr)A_{1} 
      + \cdots
      + \biggl(\sum_{\substack{I\subset \nbk{n}{1}\\ I=\{i\} }}A_i^{p-1}\biggr)A_{1},
\end{align*}
and  

\begin{align*}
  \Lc &= \biggl(\sum_{\substack{I \subset \nbk{n}{n}\\ |I|=n-2 }}A_i\biggr)^{p-1}A_{n} + 
      \biggl(\sum_{\substack{I\subset \nbk{n}{n}\\ |I|=n-4}}
      (\nlsum_{i\in I}A_i)^{p-1}\biggr)A_{n} + \cdots
      + \biggl(\sum_{\substack{I\subset \nbk{n}{n}\\|I|=2  }}A_i^{p-1}\biggr)A_{n}\\
      &+\qquad\qquad\cdots\quad\qquad +\qquad\qquad \cdots\qquad\qquad\qquad\cdots \qquad +\qquad\qquad \cdots\\
      &+ \biggl(\sum_{\substack{I \subset \nbk{n}{1}\\ |I|=n-2}}A_i\biggr)^{p-1}A_{1} + 
      \biggl(\sum_{\substack{I\subset \nbk{n}{1}\\ |I|=n-4}}
      (\nlsum_{i\in I}A_i)^{p-1}\biggr)A_{1} + \cdots
      + \biggl(\sum_{\substack{I\subset \nbk{n}{1}\\ |I|=2  }}A_i^{p-1}\biggr)A_{1}.
\end{align*}
The above expressions may be more succinctly written as
\begin{align*}
   \Rc &= \Bigl(\sum_{j=1:2:n-1}S_{j,\nbk{n}{n}}^{p-1}\Bigr)A_n 
      + \Bigl(\sum_{j=1:2:n-1}S_{j,\nbk{n}{n-1}}^{p-1}\Bigr)A_{n-1} 
      + \cdots 
      + \Bigl(\sum_{j=1:2:n-1}S_{j,\nbk{n}{1}}^{p-1}\Bigr)A_1\\
   \Lc &=\Bigl(\sum_{j=2:2:n-2}S_{j,\nbk{n}{n}}^{p-1}\Bigr)A_n 
      + \Bigl(\sum_{j=2:2:n-2}S_{j,\nbk{n}{n-1}}^{p-1}\Bigr)A_{n-1} 
      + \cdots 
      + \Bigl(\sum_{j=2:2:n-2}S_{j,\nbk{n}{1}}^{p-1}\Bigr)A_1.
   \end{align*}
For each pair of corresponding terms between $\Rc$ and $\Lc$, we can apply the statement $C_{n-1,p-1}$ because each set $\nbk{n}{k}$ is of size $n-1$. So we conclude that $\Rc \geqslant \Lc$.  

If $n$ is odd, the only difference is in the indices of the summations, which now run over $j=1\!:\!2\!:\!{n\!-\!2}$ for $\Lc$ and $j=2\!:\!2\!:\!n\!-\!1$ for $\Rc$. We conclude again that $\Rc \geqslant \Lc$, finishing the proof.
\end{proof}

\begin{corr}
  \label{cor.wolf}
  Conjecture~\ref{conj:one} is true.
\end{corr}
\begin{proof}
  Recall that for an $m\times m$ matrix $A$, $\det(A) = \wedge^m A$, where $\wedge$ denotes the usual (Grassmann) exterior product. Moreover, there exists a matrix $Z$ such that $\wedge^m A = Z^*(A^{\kron m})Z$. Since congruence preserves L\"owner order, setting $p=m$ in~(\ref{eq:6}) and transforming with $Z$, we immediately obtain inequality~(\ref{eq:9}).
\end{proof}

\noindent Using the argument of Corollary~\ref{cor.wolf} along with~\eqref{eq:8}, we obtain a more general result.
\begin{corr}
  \label{cor.dxg}
  Conjecture~\ref{conj:one} is true even when determinants are replaced by $G$-immanants.
\end{corr}

We note in passing that even more is true: combining Theorem~\ref{thm:main} with the proof technique of~\citep{linSra14} we can obtain a block-matrix version of inequality~(\ref{eq:6}). Specifically, for $1 \le i \le m$ let $\bm{A}_i=[(A_i)_{pq}]_{p,q=1}^m \geqslant 0$ be positive definite block matrices comprised of $d\times d$ complex matrices $(A_i)_{pq}$. Define $\dxg(\bm{A}) := [\dxg\bigl(A_{pq}\bigr)]_{p,q=1}^m$ for a block matrix $\bm A$. Then, Corollary~\ref{cor.dxg} holds in its ``completely positive'' incarnation applied to a collection of block matrices $\bm{A}_1,\ldots,\bm{A}_n$. We leave the details as an exercise for the interested reader.

\section{From Popoviciu to Hlawka}
\label{sec:pop}
In this section we explore the connection of Popoviciu type inequalities alluded to in the introduction. In particular, we follow the proof technique of Theorem~\ref{thm:main} to establish several Popoviciu type inequalities, one of which recovers the multivariable $G$-immanant ``superadditivity'' inequality of~\citep[Theorem~4.1]{fzhang14} as a special case.

To simplify notation, we will frequently drop subscripts on summations; hence $\sum$ is understood to mean  $\sum\nolimits_{i=1}^n$ or $\sum\nolimits_{k=1}^n$, the choice being clear from context.

For a convex function $f:\mathbb R\to\mathbb R$ and scalars $x_1,\ldots,x_k$  Jensen's inequality says that
 \begin{equation}
   \label{eq:convex}
   {f(x_1)+\cdots+f(x_k)}\geqslant kf\bigl(\dfrac{x_1+\cdots+x_k}k\bigr).
 \end{equation}
After Jensen's inequality, Popoviciu's inequality may be considered as the next-to-simplest inequality for convex functions. We restate it here.
\begin{prop}
  \label{prop:pop1}
  If  $f$ is a convex function on a real interval $I$ and $x_1,x_2,x_3 \in I$, then
  \begin{equation}
    \label{eq:pop1}
 f(x_1)+f(x_2)+f(x_3)+3f\bigl(\tfrac{x_1+x_2+x_3}{3}\bigr)\geqslant
2\bigl( f\bigl(\tfrac{x_1+x_2}{2}\bigr) +f\bigl(\tfrac{x_1+x_3}{2}\bigr) +f\bigl(\tfrac{x_2+x_3}{2}\bigr)\bigr) 
  \end{equation}  
\end{prop}

Formally, inequality~\eqref{eq:pop1} resembles Hlawka's inequality (up to scaling factors, which are actually crucial). This resemblance motivates us to examine if some known generalizations to Popoviciu's inequality for scalars, also extend to positive operators. 

We begin with the following generalization of~\eqref{eq:pop1} given by~\citet{vasc}.
\begin{prop}
  \label{prop:pop2}
  Let $f$ be convex on a real interval $I$, and $x_1,x_2, ...,x_n \in I$. Then,    \begin{equation}
    \label{eq:pop2}
    f(x_1) + \cdots + f(x_n) + \frac{n}{n-2}f\Bigl(\frac{x_1+\cdots+x_n}{n}\Bigr)
    \geqslant
    \frac{2}{n-2}\sum_{i<j}f\Bigl(\frac{x_i+x_j}{2}\Bigr). 
  \end{equation}
\end{prop}
Comparing the case $n=3$~\eqref{eq:pop1} with the Hlawka inequality~\eqref{eq:10}, it is clear that $kf\Bigl(\dfrac{x_1+\cdots+x_k}{k}\Bigr)$ should correspond to $\Vert{a_1+\cdots+a_k}\Vert$. In terms of tensor sums, after multiplying with $(n-2)$, we are led to conjecture~\eqref{eq:pop_hl2}, which turns out to be true.

\begin{theorem}
  \label{thm:pop_hl2}
  Let $A_1,\dots,A_n$ be positive definite operators. Then for each integer $p \geqslant 1$,
  \begin{equation}
    \label{eq:pop_hl2}
   (n-2) \sum A_i^p+ \left  (\sum A_i\right)^p \geqslant \sum_{i< j} (A_i+A_j)^p.
  \end{equation}
\end{theorem}
\begin{proof}
  The proof is similar to the one of Theorem~\ref{thm:one}. We proceed by induction on $p$. For $p=1$, both sides of~\eqref{eq:pop_hl2} are equal to $(n-1) \sum A_i$; for $p=2$ we again have equality, since both sides of~\eqref{eq:pop_hl2} are equal to
  \begin{align*}
    (n-1)\sum A_i^2+ \sum_{i< j} (A_iA_j+A_jA_i).
\end{align*}
Assume for the inductive step that~\eqref{eq:pop_hl2} holds for some integer $p \geqslant 2$. Then for $p+1$, 
\begin{align*}
  &\bigl(\sum A_i\bigr)^{p+1} + (n-2) \sum A_i^{p+1}\\
  &=\underbrace{\left  (\sum A_i\right)^p\Bigl(\sum A_k^{\phantom p}\Bigr)+(n-2) \sum A_i^{p}\sum A_k^{\phantom p}} -(n-2) \sum_{i< j} (A_i^pA_j^{\phantom p}+A_j^pA_i^{\phantom p}) \\
  &\geqslant \qquad\qquad\sum_{i< j} (A_i+A_j)^p\sum A_k^{\phantom p}\ \quad\qquad  \qquad -(n-2) \sum_{i< j} (A_i^pA_j^{\phantom p}+A_j^pA_i^{\phantom p}) \\
  &=   \sum_{i< j} (A_i+A_j)^{p+1}+\sum_{i< j} (A_i+A_j)^p\sum_{k\not\in\{ i,j\}} A_k^{\phantom p} -(n-2) \sum_{i< j} (A_i^pA_j^{\phantom p}+A_j^pA_i^{\phantom p}) \\
  &\geqslant  \sum_{i< j} (A_i+A_j)^{p+1}+\sum_{i< j} (A_i^p+A_j^p)\sum_{k\not\in\{ i,j\}} A_k^{\phantom p} -(n-2) \sum_{i< j} (A_i^pA_j^{\phantom p}+A_j^pA_i^{\phantom p}) \\
  &=   \sum_{i< j} (A_i+A_j)^{p+1}.% \\ &=RHS. 
\end{align*}
The first inequality follows from the induction hypothesis applied to the underbraced term, while the second inequality follows from  superadditivity~\ref{prop.basic}(iv). The final equality is easy to verify: Fix $i=1$, then the second term yields for each $j=2,\dots,n$ the product of $A_1^p$ with $(n-2)$ of the $A_k^{\phantom p}$'s ($k\ne 1$), so for each $k\ne 1$ the product $A_1^pA_k^{\phantom p}$ occurs $(n-2)$ times, and so it does also in the negative term. By symmetry, the same holds for all $i$.
\end{proof}

\begin{corr}
  \label{cor:pop1}
  Let $A_1,\ldots,A_n$ be positive definite, and let $\dxg$ be as in~(\ref{eq:2}). Then,
  \begin{equation*}
    (n-2)\sum\dxg(A_i) + \dxg(\sum A_i) \geqslant \sum_{i< j} \dxg(A_i+A_j).
  \end{equation*}
\end{corr}
Corollary~\ref{cor:pop1} combined with the superadditivity inequality Prop.~\ref{prop.basic}(iv) for the appropriate pairs of indices implies the following result of~\citet{fzhang14}.\begin{corr}[\protect{\citep[Theorem~4.1]{fzhang14}}]
  Let $A_1,\ldots,A_n$ and $\dxg$ be as in Corollary~\ref{cor:pop1}. Then,
  \begin{equation*}
    \dxg(A_1+\ldots+A_n) \geqslant \sum_{i\neq j} \dxg(A_i+A_j) - (n-2)\dxg(A_i)\qquad\text{for each } i=1,\ldots,n.
  \end{equation*}
\end{corr}

Before stating the most general result in this direction, we mention an intermediate generalization of  Popoviciu's inequality, which we call \emph{Popoviciu-Cirtoaje-Zhao inequality}.\footnote{This scalar inequality was proposed by Yufei Zhao (username Billzhao) and Vasile Cirtoaje (username Vasc) on the website AoPS~\citep{vasc} and was soon proved by Darij Grinberg in the same thread.} It states the following:

\begin{prop}
  \label{prop:pop3} 
 If  $f$ is a convex function on a real interval $I$ and $x_1,x_2, ...,x_n \in I$, then for $2\leqslant m<n$,

  \begin{equation}
    \label{eq:pop3}
%  \begin{align}
  \binom{n-2}{m-1}  \left( f(x_1)+\cdots+f(x_n)\right)  
  +n\binom{n-2}{m-2}f\Bigl(\dfrac{x_1+\cdots+x_n}{n}\Bigr) \\  
  \geqslant m\sum_{ i_1   <\cdots < i_m }  f\Bigl(\dfrac{x_{i_1}+\cdots+x_{i_m}}{m}\Bigr)  
%   \end{align}
    \end{equation} 
\end{prop}
The corresponding generalization of Theorem~\ref{thm:pop_hl2} is
\begin{theorem}
  \label{thm:pop_hl3}
  Let $A_1,\dots,A_n$ be positive definite operators. Then for each integer $p \geqslant 1$,
  \begin{equation}
    \label{eq:pop_hl3}
   \binom{n-2}{m-1} \sum A_i^p+\binom{n-2}{m-2} \left  (\sum A_i\right)^p 
   \geqslant \sum_{i_1<\cdots < i_m} (A_{i_1}+\cdots+A_{i_m})^p.
  \end{equation}
\end{theorem}

Instead of proving Theorem~\ref{thm:pop_hl3}, we move on to the most general Popoviciu type inequality for tensors, motivated by a scalar case partially treated in~\citep{aops}.  
\begin{theorem}
  \label{thm:pop_hl4}
  Let $A_1,\dots,A_n$ be positive definite operators. Let $S_{k,[n]}^p$ be defined as in \eqref{eq:7}. Then for integers $ 1\leqslant k<\ell< m\leqslant n$, 
  \begin{equation}
    \label{eq:pop_hl4}
    \frac{m-\ell}{k\binom nk} S_{k,[n]}^p +\frac{\ell-k}{m\binom nm} S_{m,[n]}^p 
    \geqslant\dfrac{m-k}{\ell\binom n\ell} S_{\ell,[n]}^p.
  \end{equation}
\end{theorem}
We omit the proof for brevity; it can be obtained by following the inductive technique developed above. It should be mentioned that the corresponding inequality for convex functions only holds for certain choices of $k,\ell,m$~\citep{aops}. 

This example shows again that the operator inequalities are weaker than the corresponding ones vor vectors or convex functions.

In the opposite direction, we might ask whether the ``convex analogue'' of Theorem~\ref{thm:main} holds:
\begin{conj}
\label{conj:thm_main}
If $f$ is a convex function on a real interval $I$ and $x_1,x_2, ...,x_n \in I$, then
\begin{equation}
\label{eq:hl_pop}
% \begin{align*}
\sum f(x_i)+3\nlsum_3 f\Bigl(\frac{x_i+x_j+x_k}3\Bigr)+\cdots\\
\geqslant 2\sum_{i<j} f\Bigl(\frac{x_i+x_j }2\Bigr)+4\nlsum_4 f\Bigl(\frac{x_i+x_j+x_k+x_l}4\Bigr)+\cdots \\
% \end{align*}
\end{equation}
\end{conj}
But this fails for $n=4$ with $f(x)=|x|$ and $(x_i)=(-10,1,1,9)$.

\subsection*{Acknowledgments}
We are thankful to the MathOverflow site for bringing the authors together to work on this paper (see \href{http://mathoverflow.net/q/182181}{mathoverflow.net/q/182181}).

\bibliographystyle{abbrvnat}

\begin{thebibliography}{13}
\providecommand{\natexlab}[1]{#1}
\providecommand{\url}[1]{\texttt{#1}}
\expandafter\ifx\csname urlstyle\endcsname\relax
  \providecommand{\doi}[1]{doi: #1}\else
  \providecommand{\doi}{doi: \begingroup \urlstyle{rm}\Url}\fi

\bibitem[Bhatia(2007)]{bhatia07}
R.~Bhatia.
\newblock \emph{{Positive Definite Matrices}}.
\newblock Princeton University Press, 2007.

%\bibitem[Berndt and Sra(2014)]{arxiv}
% W.~Berndt and S.~Sra.
%\newblock {Generalized Hlawaka-like inequalities on positive definite tensors}
%\newblock \emph{arXiv:1411.0065}, 2014.
 

\bibitem[Cirtoaje(2005)]{vasc}
V.~Cirtoaje.
\newblock {} \href{http://www.artofproblemsolving.com/Forum/viewtopic.php?p=143893#p143893} {Post in the online problem solving forum AoPS}.

\bibitem[Grinberg(2004)]{aops}
D.~Grinberg
\newblock Generalized Popoviciu inequalities.
\href{http://www.artofproblemsolving.com/Forum/viewtopic.php?f=53&t=98832}
{Post in the online problem solving forum AoPS}
 
\bibitem[Fechner(2014)]{fechner}
W.~Fechner.
\newblock Hlawka’s functional inequality.
\newblock \emph{Aequationes Mathematicae}, 87:\penalty0 71--87, 2014.

\bibitem[Freudenthal(1963)]{freudenthal}
H.~Freudenthal.
\newblock Problem 141.
\newblock \emph{Wisk. Opgaven}, 21:\penalty0 137--139, 1963.

\bibitem[Hornich(1942)]{hornich}
H.~Hornich.
\newblock Eine Ungleichung für Vektorlängen.
\newblock \emph{Mathematische Zeitschrift}, 48:\penalty0 268--274, 1942.

\bibitem[Li and Zaharia(2002)]{ckli}
C.-K. Li and A.~Zaharia.
\newblock Induced operators on symmetry classes of tensors.
\newblock \emph{Trans. Amer. Math. Soc.}, 354:\penalty0 807--836, 2002.

\bibitem[Lin(2014)]{mlin}
M.~Lin.
\newblock {A Determinantal Inequality for Positive Definite Matrices}.
\newblock \emph{Preprint}, 2014.
\href{http://web.uvic.ca/~linm/DET1.pdf}{Preprint from the author's website.}

\bibitem[Lin and Sra(2014)]{linSra14}
M.~Lin and S.~Sra.
\newblock {Complete strong superadditivity of generalized matrix functions}.
\newblock \emph{arXiv:1410.1958}, 2014.
\newblock {\it Submitted.}

\bibitem[Marcus(1973)]{marcusv1}
M.~Marcus.
\newblock \emph{Finite dimensional multilinear algebra}, volume~I.
\newblock Marcel Dekker, 1973.

\bibitem[Merris(1997)]{merris}
R.~Merris.
\newblock \emph{{Multilinear Algebra}}.
\newblock Gordon \& Breach, Amsterdam, 1997.

\bibitem[Mitrinovi{\'c}(1970)]{mitrinovic}
D.~S. Mitrinovi{\'c}.
\newblock \emph{{Analytic Inequalities}}.
\newblock Springer-Verlag, Berlin, 1970.

\bibitem[Niculesu and Persson(2006)]{niculesu06}
C.~Niculesu and L.~E. Persson.
\newblock \emph{Convex functions and their applications: a contemporary
  approach}, volume~13 of \emph{Science \& Business}.
\newblock Springer, 2006.

\bibitem[Paksoy et~al.(2014)Paksoy, Turkmen, and Zhang]{fzhang14}
V.~Paksoy, R.~Turkmen, and F.~Zhang.
\newblock Inequalities of generalized matrix functions via tensor products.
\newblock \emph{Electron. J. Linear Algebra}, 27:\penalty0 332--341, 2014.

\bibitem[Radulescu and Radulescu(1996)]{radu}
M.~Radulescu and S.~Radulescu.
\newblock {Generalizations of Dobrushin's Inequalities and Applications}.
\newblock \emph{Journal of Mathematical Analysis and Applications},
  204\penalty0 (3):\penalty0 631--645, 1996.

\bibitem[Tie et~al.(2011)Tie, Cai, and Lin]{tie11}
L.~Tie, K.-Y. Cai, and Y.~Lin.
\newblock {Rearrangement inequalities for Hermitian matrices}.
\newblock \emph{Linear Algebra and its Applications}, 434:\penalty0 443--456,
  2011.

\end{thebibliography}
\setlength{\bibsep}{2pt}

\end{document}